\documentclass[12pt,a4paper]{amsart}
\usepackage[top=1.11in, bottom=1.11in, left=1.13in, right=1.13in]{geometry}
\usepackage{amsthm, amsmath, amssymb, amscd, latexsym, multicol, verbatim, enumerate, graphics,xypic}
\usepackage{tikz}
\usepackage{hyperref}
\usepackage[english]{babel}

\makeatother \theoremstyle{remark}
\numberwithin{equation}{section}

\theoremstyle{definition} 

\newtheorem{theorem}{Theorem}[section]
\newtheorem{definition}[theorem]{Definition}\theoremstyle{definition}
\newtheorem{proposition}[theorem]{Proposition}
\newtheorem{lemma}[theorem]{Lemma}
\newtheorem{corollary}[theorem]{Corollary}
\newtheorem{remark}[theorem]{Remark}
\newtheorem{example}[theorem]{Example}
\newtheorem{conjecture}[theorem]{Conjecture}
\newtheorem*{thm}{Theorem}
\newtheorem*{conj}{Conjecture}

\newcommand{\ra}{\rightarrow}

\newcommand{\bs}{\mathbf{s}}
\newcommand{\bt}{\mathbf{t}}
\newcommand{\bm}{\mathbf{m}}

\newcommand{\lie}{\mathfrak}

\newcommand{\St}{\operatorname*{St}}
\newcommand{\Cone}{\operatorname*{Cone}}

\newcommand{\blambda}{\mathbf{\lambda}}
\newcommand{\bmu}{\mathbf{\mu}}
\newcommand{\bz}{\mathbb{Z}}

\begin{document}

\title{Marked chain-order polytopes}
\author{Xin Fang, Ghislain Fourier}
\address{\newline Mathematisches Institut, Universit\"at zu K\"oln, Germany} 
\email{xinfang.math@gmail.com}
\address{\newline School of Mathematics and Statistics, University of Glasgow, UK}
\email{ghislain.fourier@glasgow.ac.uk}
\keywords{Marked order polytopes, Marked chain polytopes, Marked chain-order polytopes, Ehrhart equivalence}

\begin{abstract}
We introduce in this paper the marked chain-order polytopes associated to a marked poset, generalizing the marked chain polytopes and marked order polytopes by putting them as extremal cases in an Ehrhart equivalent family. Some combinatorial properties of these polytopes are studied. This work is motivated by the framework of PBW degenerations in representation theory of Lie algebras.
\end{abstract}

\maketitle
\section*{Introduction}

\subsection*{Marked chain and order polytopes}
For a finite poset $(P, \prec)$, Stanley \cite{Sta86} introduced two polytopes associated to $P$, the chain polytope and the order polytope, both defined in $\mathbb{R}_{\geq 0}^{|P|}$. The order polytope is defined along the order relations in $P$: for $p \prec q$ the coordinates $x_p$ and $x_q$ corresponding to $p$ and $q$ satisfy $x_p \leq x_q$. 
One may think of the cover relations of the poset as defining hyperplanes of a cone and the order polytope is the restriction of the cone to the $|P|$-dimensional cube of volume $1$. 
The chain polytope is defined by the chains in the poset, \emph{e.g.} a chain $p_1 \prec \ldots \prec p_s$ in $P$ gives rise to the relation $x_{p_1} + \ldots + x_{p_s} \leq 1$. Both are interesting examples for 0-1 polytopes and gained, especially recently, some attention (see for example \cite{HL12a, HL12b, HMT15}).
\par
These notions have been generalized by Ardila-Bliem-Salazar \cite{ABS11} to marked posets. Let $A \subset P$ be a subset containing at least all minimal and all maximal elements, and $\blambda \in \mathbb{Z}_{\geq 0}^{|A|}$ be a marking vector.  
The \textit{marked order polytope} $\mathcal{O}_{P,A}(\blambda)$ is the restriction of the cone to a cuboid defined by $\blambda$, while for the \textit{marked chain polytope}  $\mathcal{C}_{P,A}(\blambda)$ the defining hyperplanes (given by chains) are translated via $\blambda$. 
In both cases, the chain and the order polytopes are obtained as special cases by adding a smallest and largest element to $P$, marked by $0$ and $1$.
\par
It has been shown in \cite{ABS11} that the marked order and the marked chain polytope are Ehrhart equivalent for a fixed triple $(P, A, \blambda)$. Further, necessary and sufficient conditions on the poset have been provided for order and chain polytopes (resp. marked order and marked chain polytopes) to be unimodular equivalent \cite{HL12a, Fou15a}.

\subsection*{Motivations from representation theory}
We are in particular interested in polytopes, whose lattice points parametrize (monomial) bases for representations of semi-simple Lie algebras. Therefore our main and motivating example is the poset $P_n$, where the vertices are $\{p_{i,j}\mid 0\leq i\leq j\leq n\}$ and the cover relations for any $1\leq i\leq j\leq n$ are: $p_{i-1,j}\prec p_{i,j}\prec p_{i-1,j-1}$ (see Example~\ref{gt-pattern} for the $n=4$ case). 
We fix the linearly ordered marking set $A_n = \{ p_{0,0},\, p_{0,1},\, \ldots,\, p_{0,n} \},$ with marking $\lambda_k = \lambda_{p_{0,k}}$, so $\blambda=(\lambda_0 \geq \lambda_1\geq \ldots \geq \lambda_n)$ is a partition.
\par
The order polytope $\mathcal{O}_{P_n,A_n}(\blambda)$ is known as the \textit{Gelfand-Tsetlin polytope} (\cite{GT50}), whose lattice points parametrize a monomial basis of the simple $\lie{sl}_{n+1}$-module $V(\blambda)$ of highest weight $\blambda$. It has been shown \cite{FFoL11a} that the lattice points in  $\mathcal{C}_{P_n,A_n}(\blambda)$ parametrize a monomial basis in the PBW-graded module $\operatorname{gr } V(\blambda)$. 
This module is the associated graded module for the natural PBW-degree filtration on the universal enveloping algebra of $\lie{sl}_{n+1}$. Similar results on marked poset polytopes are known for the symplectic Lie algebra \cite{FFoL11b}, certain Demazure modules for $\lie{sl}_n$ \cite{Fou14, BF15}. 
\par
Considering the corresponding associated graded algebras and modules provides an interesting bridge between Lie theory and commutative algebra (see \cite{FFR15} for another point of view via quantum groups on this), and further induces (toric) degenerations of the complete flag variety (for more details see  \cite{Fei11, Fei12, FFoL13b}).
\par
An attempt to generalize these PBW-degenerations has been taken recently in \cite{CIFFFR16}, where linear degenerations of the complete flag variety are studied and it turns out that the space of global sections of line bundles on certain of these degenerations (the PBW locus) can be considered as associated graded modules $V^{\mathbf{i}}(\lambda)$ (with an appropriate filtration on the universal enveloping algebra) of $V(\lambda)$. It is natural to ask for parametrizations of bases in $V^\mathbf{i}(\lambda)$ via polytopes.

\subsection*{Marked chain-order polytopes}
In this paper we introduce a new family of polytopes associated to a marked poset $(P, A, \blambda)$, which in the special case of $P_n$ conjecturally parametrize a monomial basis of $V^{\mathbf{i}}(\lambda)$. 
\par
For this let $U_1\cup U_2 = P \setminus A, U_1 \cap U_2= \emptyset$ be a decomposition of the non-marked elements of $P$, we call the decomposition \emph{admissible} if there is no element in $U_1$ covered by an element in $U_2$.
\par
For such an admissible decomposition we define the \textit{marked chain-order polytope} $\mathcal{CO}_{U_1, U_2}(\blambda)$ by considering order conditions for $U_1$ and then regard $A \cup U_1$ as the marked subset for the chain part $U_2$, (for details see Definition~\ref{def:chain-order}).
\par
We restrict ourselves here for the sake of representation theory to admissible decompositions. Polytopes associated to an arbitrary decomposition (called \textit{layered marked chain-order polytopes}) will be studied in a forthcoming publication \cite{FF15}. 
It should be pointed out that these polytopes are fundamentally different from the order-chain polytopes defined in \cite{HLLMT15}, see Remark~\ref{remHLLMT} for details.
\par
Let us assume that $(P, A, \blambda)$ is regular (see Definition~\ref{def:regular}), \emph{e.g.} there are no nontrivial redundancies among the defining inequalities. 
We call $p \in P \setminus A$ a \emph{star element} if there are at least two elements in $P$ covering $p$ and there are at least two maximal chains in $P$ having $p$ as their largest element. The set of all star elements is denoted by $\St(P)$.
The first main result is

\begin{thm} 
Let $(U_1, U_2)$ and $(V_1, V_2)$ be two admissible decompositions with $U_1\subset V_1$. Then the polytopes $\mathcal{CO}_{U_1, U_2}(\blambda)$ and $\mathcal{CO}_{V_1, V_2}(\blambda)$ are unimodular equivalent if and only if 
$$U_1 \cap \St(P) = V_1 \cap \St(P).$$
If moreover $U_1 \cap \St(P) = (V_1\cap\St(P)) \cup \{p\}$ for $p\in \St(P)$, then the number of facets in $\mathcal{CO}_{V_1, V_2}(\blambda)$ is strictly less than the number of facets in $\mathcal{CO}_{U_1, U_2}(\blambda)$.
\end{thm}

We would like to state the following conjecture, a generalization of a conjecture by Stanley, Hibi and Li for chain and order polytopes and by the second author for marked chain and marked order polytopes  \cite{Sta86, HL12a, Fou15a}:

\begin{conj}
Let $(U_1, U_2)$ and $(V_1, V_2)$ be two admissible decompositions such that $U_1 \cap \St(P) = (V_1\cap\St(P)) \cup \{p\}$ for $p\in \St(P)$. Then for any $0 \leq i \leq |P\setminus A|$, the number of $i$-dimensional faces in $\mathcal{CO}_{V_1, V_2}(\blambda)$ is greater or equal to the number of $i$-dimensional faces in $\mathcal{CO}_{U_1, U_2}(\blambda)$.
\end{conj}

We prove parts of the conjecture, namely we show the conjectured inequality for the number of facets (Corollary~\ref{Cor:SHLfacet}).
\par
We denote 
$$S_{U_1, U_2}(\blambda) = \mathcal{CO}_{U_1, U_2}(\blambda) \cap \bz_{\geq 0}^{|P\setminus A|}$$
the lattice points in the marked chain-order polytope.

\begin{thm}
Let $(P, A, \blambda)$ be a marked poset and $(U_1, U_2)$ an admissible decomposition. Then the marked chain-order polytope $\mathcal{CO}_{U_1, U_2}(\blambda)$ is a normal lattice polytope. 
Moreover, if $A$ is linearly ordered, then $S_{U_1, U_2}(\blambda_1 + \blambda_2) = S_{U_1, U_2}(\blambda_1)  + S_{U_1, U_2}(\blambda_2)$ for any markings $\blambda_1$ and $\blambda_2$.
\end{thm}

Since $\mathcal{CO}_{U_1, U_2}(\blambda)$ is a lattice polytope, there exists a counting polynomial, the Ehrhart polynomial, whose value at $N\in\mathbb{N}$ gives the number of lattice points in $\mathcal{CO}_{U_1, U_2}(N\blambda)$.

\begin{thm}
Let $(P, A, \blambda)$ be a marked poset. Then all marked chain-order polytopes associated to admissible decompositions are Ehrhart equivalent, \emph{i.e.} they have the same Ehrhart polynomial.
\end{thm}

\subsection*{Applications and conjectures.}
We turn back to the representation theory and the poset $P_n$. Starting from the marked order polytope, we have a poset of \textit{degenerations} of this polytope, whose most degenerated element is the marked chain polytope. 
This poset of degenerations conjecturally corresponds to the poset of  linearly degenerate flag varieties in the PBW locus and lattice points in the marked chain-order polytope parametrize a monomial basis of the degenerated module. This is shown to be true for the marked order polytope (\cite{GT50}) and the marked chain polytope (\cite{FFoL11a, FFoL13a}). It would be very interesting to understand how this fits into the more general framework of toric degenerations in \cite{FFL15, FFL16}. For more details on the linear degenerations of the flag varieties we refer to \cite{CIFFFR16}.
\par
In \emph{loc.cit} and  \cite{CIL14, CILL15}, the modules $V^{\mathbf{i}}(\lambda)$ have been identified with Demazure modules $V_{w^{\mathbf{i}}}(\blambda^{\mathbf{i}})$ of a certain $\lie{sl}_k$, where $w^{\mathbf{i}} \in \mathfrak{S}_{k}$ is an element in the Weyl group and $\blambda^{\mathbf{i}}$ is a weight obtained from $\lambda$.  Using the crystal graph and Kashiwara's root operators, Littelmann \cite{Lit98} has provided for any reduced expression $\underline{w}^{\mathbf{i}}$  in terms of simple reflections and any highest weight $\lambda^\mathbf{i}$ a convex polytope $Q_{\underline{w}^{\mathbf{i}}}(\blambda^{\mathbf{i}})$, the string polytope associated to the Demazure module $V_{w^{\mathbf{i}}}(\blambda^{\mathbf{i}})$, whose lattice points parametrize a monomial basis of the Demazure module.
\par
The marked order polytope is known to be a string polytope \cite{Lit98}; the marked chain polytope is also proved to be a string polytope \cite{FoL15}. We conjecture that any marked chain-order polytope (of the fixed poset $P_n$) is a string polytope corresponding to certain reduced expression and highest weight. This conjecture is true for $n \leq 5$ by direct verification using polymake (\cite{GJ97}).

\subsection*{Organization of paper}

The paper is organized as follows: in Section~\ref{sec2} we fix notations, introduce marked chain-order polytopes and recall results on marked chain and marked order polytopes. In Section~\ref{sec3} we provide the proofs on the normality and Minkowski property, while in Section~\ref{sec4} we prove the Ehrhart equivalence. Section~\ref{sec5} is on the number of facets and equivalent classes of polytopes; Section~\ref{sec6} gives the refinement of the Stanley-Hibi-Li conjecture. Finally, Section~\ref{sec7} gives the example of the Gelfand-Tsetlin poset.\\

\noindent
\textbf{Acknowledgements.}
The work of X.F. is supported by the Alexander von Humboldt
Foundation. The work of G.F. has been funded by the DFG priority program 1388 "Representation Theory". We would like to thank Katharina Jochemko, Christian Haase and Peter Littelmann for helpful and fruitful discussions and remarks.

\section{Definition}\label{sec2}

We recall here the definition of a marked poset, due to \cite{ABS11}, but set into a wider context. Throughout this paper, all posets are assumed to be finite. For a finite set $S$, we let $\#S$ or $|S|$ denote its cardinal.
\par
For the convenience of the reader, we recall some definitions on polytopes.

\subsection{Polytopes}
A \emph{polytope} is the convex hull of a finite number of points in some $\mathbb{R}^d$. Let $P$ be a polytope in $\mathbb{R}^d$. We recall several definitions around a polytope:
\begin{enumerate}
\item $P$ is called a \emph{lattice polytope} if its vertices have integral coordinates.
\item For $c\in\mathbb{R}$, the \emph{dilation} of the polytope by $c$ is $cP:=\{cx\mid x\in P\}$.
\item The \emph{lattice points} in $P$ are defined by $S(P):=P\cap\mathbb{Z}^d$. 
\item Let $P$, $Q$ be two polytopes in $\mathbb{R}^d$, the Minkowski sum of $P$ and $Q$ is defined by
$$P+Q=\{p+q\mid p\in P,\ q\in Q\}.$$
\item A lattice polytope $P$ is said to be \emph{normal}, if for any $n\in\mathbb{N}$, $S(nP)$ is the $n$-fold Minkowski sum of $S(P)$.
\end{enumerate}

For any lattice polytope $P$ there exists a polynomial $E_P(t)$, the \emph{Ehrhart polynomial} of $P$, satisfying: $\# S(nP)=E_P(n)$. Two lattice polytopes $P$ and $Q$ are called \emph{Ehrhart equivalent} if they have the same Ehrhart polynomial, consequently Ehrhart equivalent polytopes have the same number of lattice points.

\subsection{Notations on posets}
Let $(P, \prec) $ be a poset. For $p,q\in P$, we say \emph{$p$ covers $q$}, denoted by $q\ra p$, if $q\prec p$ and for any $r\in P$ satisfying $q\prec r\prec p$ we have $q=r$ or $r=p$. We denote $\ra p=\{q\in P\ |\ q\ra p\}$ the set of elements covered by $p$ and $p\ra =\{q\in P\ |\ p\ra q\}$ the set of elements covering $p$.
\par
For $p\prec q\in P$, a \emph{$(p,q)$-chain} is a sequence $q_1,q_2,\ldots,q_s\in P$ such that $p \prec q_1\prec q_2\prec\ldots\prec q_s\prec q$; it is called a \emph{maximal $(p,q)$-chain} if $p\ra q_1\ra q_2\ra\ldots\ra q_s\ra q$. We denote $p\rightsquigarrow$ (resp. $\rightsquigarrow p$) the set of maximal chains starting in $p$ (resp. ending in $p$).
 \begin{definition} An element $p\in P$ is called a \emph{star element} if $\# p\ra\, \geq 2$ and $\# \rightsquigarrow p \,\geq 2$. That is to say, there exist at least two distinct elements covering $p$ and two distinct maximal chains ending in $p$. The set of all star elements in $P$ is denoted by $\St(P)$.
\end{definition}

\begin{example}
Consider the following poset $P$ where $1$ and $2$ are minimal elements:\\
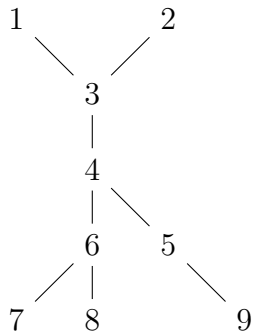
\begin{center}
\begin{tikzpicture}
  \node (a) at (-1,1) {$1$};
  \node (b) at (1,1) {$2$};
  \node (c) at (0,0) {$3$};
  \node (d) at (0,-1) {$4$};
  \node (e) at (1,-2) {$5$};
  \node (f) at (0,-2) {$6$};
  \node (g) at (-1,-3) {$7$};
  \node (h) at (0,-3) {$8$};
  \node (i) at (2,-3) {$9$};
  \draw (a) -- (c) -- (b) 
  (c) -- (d) -- (e) -- (i)
  (d) -- (f) -- (g) 
  (f) -- (h);
\end{tikzpicture}
\end{center}
Then $\St(P)=\{3\}$, $\ra 3=\{1,2\}$, $3\ra=\{4\}$ and
$$3\rightsquigarrow =\{3\prec 4\prec 5\prec 9,\ 3\prec 4\prec 6\prec 7,\ 3\prec 4\prec 6\prec 8\}.$$
\end{example}

\subsection{Admissible decompositions}
\begin{definition}
A \emph{decomposition} of $P$ is a pair of subsets $(U_1,U_2)$ of $P$ such that $U_1\cap U_2=\emptyset$ and $U_1\cup U_2=P$; it is called \emph{admissible} if furthermore there do not exist $u_1\in U_1$ and $u_2\in U_2$ such that $u_1\prec u_2$.
\end{definition}

\begin{example} Here are some examples:
\begin{enumerate}
\item $U_1 = P, U_2 = \emptyset$ or $U_1 = \emptyset, U_2=P$ are the both extremal cases.
\item In the example above, $U_1= \{4,7,8,9\}$, $U_2 = \{ 1,2,3,6,5\}, $ is a decomposition but it is not admissible, since $ 4\prec 5$.
\end{enumerate}
For a linearly ordered poset $P$, the number of different admissible decompositions is $\#P+1$.
\end{example}
We can associate to any decomposition $(U_1, U_2)$ a polyhedral cone $\Cone(U_1, U_2) $:
\begin{small}
$$
\left\{ (x_p) \in \mathbb{R}^{|P|} \mid \begin{cases} x_p \geq 0, &\text{for } p \in U_2; \\ x_p \leq x_q, \; & \text{for all } p \prec q \in U_1; \\ x_{p_1} + \ldots + x_{p_s} \leq x_q, \; & \text{for all } \, p_1 \prec  \ldots \prec p_s \prec q,\ q \in U_1,\ p_i \in U_2.  \end{cases}\right\}
$$
\end{small}

\subsection{Marked posets and associated polytopes}
An element in $P$ is called \emph{extremal} if it is either a maximal or a minimal element. Let $A$ be a subset of $P$ containing at least all extremal elements of $P$. A \emph{marking} of $A$ is a map $\blambda:A\ra\mathbb{Z}_{\geq 0}$, it can be looked as a vector in $\mathbb{Z}_{\geq 0}^{|A|}$; for $a\in A$, we will denote $\blambda_a:=\blambda(a)$. A \emph{marked poset} (\cite{ABS11}) is such a triple $(P, A, \blambda)$.
\par
A decomposition (resp. an admissible decomposition) of the marked poset $(P,A,\blambda)$ is a decomposition (resp. an admissible decomposition) of the subposet $P\backslash A$.
\par
A \emph{$U_2$-chain} in $P$ is a set $\mathbf{c} = \{a, p_1, \ldots, p_n, b\}$ with $n \geq 1$, $a,b \in A \cup U_1$ and $p_i \in U_2$ such that
$$ b\prec p_n\prec \ldots\prec p_1\prec a. $$
We denote by $\mathbb{D}_{U_2}$ the set of all $U_2$-chains. For a $U_2$-chain $\mathbf{p} \in \mathbb{D}_{U_2}$, we denote $\mathbf{p}_a$ the maximal element in the chain and $\mathbf{p}_b$ the minimal element in the chain, so $\mathbf{p}_a, \mathbf{p}_b \in A\cup U_1$.
\par
Although some definitions and results in this paper remain to be true for an arbitrary decomposition $(U_1,U_2)$ of $(P,A,\blambda)$, we restrict ourselves to the admissible case.

\begin{definition} \label{def:chain-order}
Let $(P, A, \blambda)$ be a marked poset and $(U_1, U_2)$ be an admissible decomposition. We associate to it the \textit{marked chain-order polytope} $\mathcal{CO}_{U_1, U_2}(\blambda)$:
\begin{small}
$$
\left\{ (x_p) \in \Cone(U_1, U_2)\mid 
\begin{cases}  
\text{ for } p\prec a,\ p\in U_1,\ a\in A: &  x_p \leq \blambda_a; \\ 
\text{ for } b\prec q,\ q\in U_1,\ b\in A: &  \blambda_b \leq x_q; \\ 
\text{ for } p\prec q \in U_1: &  x_p \leq x_q; \\ 
\text{ for } p \in U_2: &  x_p \geq 0; \\
\text{ for any} \; \mathbf{p} \in \mathbb{D}_{U_2}: & \sum_{q \in \mathbf{p} \cap U_2} x_q \leq \blambda_{\mathbf{p}_a} - \blambda_{\mathbf{p}_b},
\end{cases} \right\}
$$
where $\lambda_q := x_q$ if $q \in U_1$ in the last inequality. We further denote the set of lattice points in this polytope
$$
S_{U_1, U_2}(\blambda) := \mathcal{CO}_{U_1, U_2}(\blambda)  \cap \bz_{\geq 0}^{|P\setminus A|}.
$$
\end{small}
\end{definition}
\begin{example}
Consider the following poset $P$ and $A=\{p_{0,0},p_{0,1},p_{0,2}\}$. We fix a marking $\blambda:A\ra\mathbb{Z}_{\geq 0}$ by: $\blambda(p_{0,0})=\blambda_0$, $\blambda(p_{0,1})=\blambda_1$, $\blambda(p_{0,2})=\blambda_2$.
\\
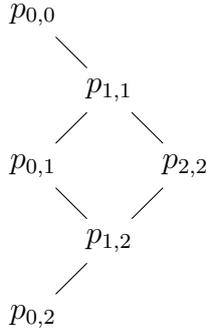
\begin{center}
\begin{tikzpicture}
  \node (a) at (0,0) {$p_{0,0}$};
  \node (b) at (0,-2) {$p_{0,1}$};
  \node (c) at (0,-4) {$p_{0,2}$};
  \node (f) at (1,-1) {$p_{1,1}$};
  \node (g) at (1,-3) {$p_{1,2}$};
  \node (j) at (2,-2) {$p_{2,2}$};

  \draw (a) -- (f) -- (j) 
  (b) -- (g)
  (b) -- (f)
  (c) -- (g) -- (j);
\end{tikzpicture}
\end{center}

Suppose $U_1 = \{p_{1,1}\}$ and $U_2= \{p_{1,2}, p_{2,2}\}$, then the marked chain-order polytope $\mathcal{CO}_{U_1, U_2}(\blambda)$ is defined by the following inequalities (we set $x_{i,j} := x_{p_{i,j}}$):
\[
 x_{1,2} \geq 0, \,   \,x_{2,2} \geq 0,
\]
\[ \lambda_0 \geq x_{1,1} \geq \lambda_1, \,   \,\lambda_1 - \lambda_2 \geq x_{1,2}, \,   \, x_{1,1} - \lambda_2 \geq x_{2,2} + x_{1,2}.
\]
\end{example}
\begin{remark}
These are special cases of layered marked poset polytopes which will be studied in a forthcoming publication \cite{FF15}, where the condition that the decomposition $(U_1, U_2)$ has to be admissible will be dropped.
\end{remark}

\begin{remark}\label{Rmk:RationalVertex}
Since $\blambda\in\mathbb{Z}_{\geq 0}^{|A|}$, the polytope $\mathcal{CO}_{U_1,U_2}(\blambda)$ is defined by inequalities with integer coefficients. It implies that its vertices have rational coordinates.
\end{remark}

\begin{remark}\label{remHLLMT}
These are \emph{not} order-chain polytopes introduced by Hibi \emph{et al.} \cite{HLLMT15}. The main difference is in the definition, we are considering decompositions of the set of vertices in the Hasse diagram while in \cite{HLLMT15}, the authors considered a decomposition of the set of edges. Here is an example to see the difference; consider the following marked poset:
\begin{center}
\begin{tikzpicture}
  \node (a) at (0,2) {$1$};
  \node (b) at (0,1) {$p$};
  \node (c) at (0,0) {$q$};
  \node (d) at (0,-1) {$r$};
  \node (e) at (0,-2) {$0$};

  \draw (a) -- (b) -- (c)-- (d) -- (e);
\end{tikzpicture}
\end{center}
For the chain-order polytope one has to fix a decomposition of the vertices, for example $U_1 = \{p\}, U_2 = \{q,r\}$. Then the chain-order polytope is defined by the inequalities
$$
x_q\geq 0,\ x_r\geq 0,\ x_p\leq 1,\ x_q+x_r\leq x_p.
$$
Then there is no order-chain polytope that is defined by the same hyperplanes. For example if we decompose $cE = \{ (q,r)\}$ and $oE = \{ (p,q) \}$. Then the order-chain polytope is defined by the inequalities
$$
x_q\geq 0,\ x_r\geq 0,\ x_p\leq 1,\ x_q+x_r\leq 1,\ x_q\leq x_p.
$$
In fact, these two polytopes are fundamentally differently defined, so there is no easy way to compare those. But important to notice here is, that our results are not valid for the order-chain polytopes, for example, they are not Ehrhart equivalent in general.
\end{remark}

\begin{example}\label{Ex:extremal}
We consider in this example two extremal cases:
\begin{enumerate}
\item $(U_1,U_2)= (P \setminus A,\emptyset)$: the corresponding marked chain-order polytope is the marked order polytope (\cite{ABS11}) and we will denote it by $\mathcal{O}_{P,A}(\blambda)$.
\item $(U_1,U_2)= (\emptyset,P \setminus A)$: the corresponding marked chain-order polytope is the marked chain polytope (\cite{ABS11}) and we will denote it by $\mathcal{C}_{P, A}(\blambda)$.
\end{enumerate}
\end{example}

These polytopes were defined and studied in \cite{ABS11} generalizing the order and chain polytopes, defined in \cite{Sta86}. Order and chain polytopes are considered as the special case where vertices in $A$ are marked only by $0$ and $1$. 
\par
The marked order and marked chain polytopes were originally introduced in the context of representation theory of semi-simple Lie algebras. We will see how the marked chain-order polytopes fit into this context and how again, representation theory motivates most of the questions we are answering in this paper.

\subsection{Properties of order and chain polytopes}
Let $(P,A,\blambda)$ be a marked poset, $\mathcal{O}_{P,A}(\blambda)$ be the associated marked order polytope and $\mathcal{C}_{P,A}(\blambda)$ be the marked chain polytope. Notice that by Example \ref{Ex:extremal},
$$S(\mathcal{O}_{P,A}(\blambda))=S_{P\setminus A,\emptyset}(\lambda)\ \text{ and }\ S(\mathcal{C}_{P,A}(\blambda))=S_{\emptyset,P\setminus A}(\lambda).$$
\par
We list some combinatorial properties of these polytopes.
\begin{theorem}\label{Thm:C/O}
The following statements hold.
\begin{enumerate}
\item (\cite{ABS11}) The marked order polytope $\mathcal{O}_{P,A}(\blambda)$ and the marked chain polytope $\mathcal{C}_{P,A}(\blambda)$ are Ehrhart equivalent.
\item (\cite{Fou15a}) Let $\blambda, \bmu \in \bz_{\geq 0}^{|A|}$ be two markings. Then we have the following \textit{Minkowski sum property} for the marked chain polytope:
$$
S_{\emptyset, P \setminus A}( \blambda) + S_{\emptyset, P \setminus A}( \bmu)  = S_{\emptyset, P \setminus A}( \blambda + \bmu). 
$$
\end{enumerate}
\end{theorem}

\subsection{Basic properties of chain-order polytopes}

We suppose in this subsection that $(P,A,\blambda)$ is a marked poset with an admissible decomposition $(U_1,U_2)$.

\begin{proposition}\label{Prop:COdialation}
For $n\in\mathbb{N}$, we have $n\,\mathcal{CO}_{U_1,U_2}(\blambda)=\mathcal{CO}_{U_1,U_2}(n\blambda)$.
\end{proposition}

\begin{proof}
It suffices to notice that in the definition of a chain-order polytope, all inequalities contain only linear terms on $\blambda_a$ for $a\in A$.
\end{proof}

We consider two projections of a chain-order polytope, \emph{e.g.} the following diagram:
\[
\xymatrix{
& \mathbb{R}^{|P\setminus A|}\ar[ld]_-{\pi_1} \ar[rd]^-{\pi_2} & \\
\mathbb{R}^{|U_1|} & & \mathbb{R}^{|U_2|}
}
\]
where $\pi_1$ and $\pi_2$ are linear projections onto the coordinates in $U_1$ and $U_2$, respectively. 
Given a polytope $Q \subset  \mathbb{R}^{|P\setminus A|}$, we consider its set of lattice points $S(Q)$. For every $\mathbf{p} \in S(Q)$, we can associate to it the set $\pi_2 \circ \pi_1^{-1} \circ \pi_1(\mathbf{p})$ (\emph{i.e.}, the $\pi_2$ image of the $\pi_1$ fiber of $\pi_1(\mathbf{p}))$.
 
\begin{definition}
The polytope $Q$ is said to have the \textit{decomposition property} with respect to $(U_1, U_2)$ if 
$$
S(Q) = \bigcup_{\mathbf{p} \in S(Q)} \{ (\pi_1(\mathbf{p}), \mathbf{q})\in\mathbb{R}^{|U_1|}\times \mathbb{R}^{|U_2|} \mid \mathbf{q} \in \pi_2 \circ \pi_1^{-1} \circ \pi_1(\mathbf{p}) \}
$$
\end{definition}
We apply these projections to the marked order, marked chain and marked chain-order polytopes. If we denote $A_1:=A\cup U_1$ and $A_2:=A\cup U_2$, then the order polytope $\mathcal{O}_{A_1,A}(\blambda)$ is a polytope in $\mathbb{R}^{|U_1|}$.

\begin{lemma}\label{Lem:pr1}
We have
$$\pi_1(\mathcal{O}_{P,A}(\blambda))=\mathcal{O}_{A_1,A}(\blambda)=\pi_1(\mathcal{CO}_{U_1,U_2}(\blambda)).$$
\end{lemma}

\begin{proof}
We first show that $\pi_1(\mathcal{O}_{P,A}(\blambda))=\pi_1(\mathcal{CO}_{U_1,U_2}(\blambda))$.
By definition, $ \pi_1(\mathcal{CO}_{U_1, U_2}(\blambda))$ is the set
$$
 \left\{ (x_p) \in \mathbb{R}^{|U_1|} \mid \begin{cases}  
\text{ for } p\prec a,\ p\in U_1,\ a\in A: &  x_p \leq \blambda_a \\ 
\text{ for } b\prec q,\ q\in U_1,\ b\in A: &  \blambda_b \leq x_q \\ 
\text{ for } p\prec q \in U_1: &  x_p \leq x_q  
\end{cases} \right\}.$$
We see by the defining relations that $\pi_1(\mathcal{O}_{P,A}(\blambda))$ is also contained in the set above. It is left to prove the equality here, for this we provide a face in $\pi_1(\mathcal{O}_{P,A}(\blambda))$ which is the preimage of $ \pi_1(\mathcal{CO}_{U_1,U_2}(\blambda)) $. For this, we consider the set $\mathcal{F}$ of $\mathcal{O}_{P,A}(\blambda)$ defined by setting for all $p \in U_2$:
$$
x_p = \operatorname{max }\ \{ \lambda_a \mid a\in A\ \text{such that}\ a\prec p \}.
$$
Since the decomposition $(U_1,U_2)$ is admissible, the set $\mathcal{F}$ is non-empty and hence a face.
Then we have $\pi_1 (\mathcal{F}) =  \pi_1(\mathcal{CO}_{U_1,U_2}(\blambda))$, which shows the equality.
\par
The identity $\mathcal{O}_{A_1,A}(\blambda)=\pi_1(\mathcal{CO}_{U_1,U_2}(\blambda))$ is clear by definition.
\end{proof}

We turn to study the projection $\pi_2$. Let $\mathbf{s}\in\mathcal{O}_{P,A}(\blambda)$ be a lattice point. It provides another marked poset $(P,A_1,\blambda^{\bs})$ where 
$$\blambda^{\bs}(p) = \begin{cases} \blambda_p, &  \text{if }\ p\in A; \\ \bs_p, & \text{if }\ p\in U_1. \end{cases}$$
Notice that if $\bs,\bt\in\mathcal{O}_{P,A}(\blambda)$ such that $\pi_1(\bs)=\pi_1(\bt)$, then $\blambda^\bs=\blambda^\bt$.
\par
Let $\mathcal{O}_{P,A_1}(\blambda^\bs)$ and $\mathcal{C}_{P,A_1}(\blambda^\bs)$ be the marked order and marked chain polytopes associated to $(P,A_1,\blambda^\bs)$. The following lemma is clear by definition.

\begin{lemma}\label{Lem:Decomposition}
The following identities hold:
\begin{enumerate}
\item for $\bs\in\mathcal{O}_{P,A}(\blambda)$, 
$$\pi_2\circ \pi_1^{-1}\circ \pi_1(\bs)=\mathcal{O}_{P,A_1}(\blambda^\bs);$$
\item for $\bs\in\mathcal{CO}_{U_1,U_2}(\blambda)$, 
$$\pi_2\circ \pi_1^{-1}\circ \pi_1(\bs)=\mathcal{C}_{P,A_1}(\blambda^\bs).$$
\end{enumerate}
\end{lemma}

\begin{remark}
Let $(V_1,V_2)$ be another admissible decomposition of the marked poset $(P,A,\blambda)$ with $U_1\subseteq V_1$. Then it is easy to see that the polytope $\mathcal{CO}_{V_1,V_2}(\blambda)$ has the decomposition property with respect to $(U_1, U_2)$.
\end{remark}

\begin{remark}
From Lemma \ref{Lem:Decomposition}, $\mathcal{O}_{P,A}(\blambda)$ and  $\mathcal{CO}_{U_1,U_2}(\blambda)$ have the decomposition property with respect to $(U_1, U_2)$:
\begin{enumerate}
\item for any lattice point $\bs\in\mathcal{O}_{P,A}(\blambda)$, there exists a unique lattice point $\bt\in\mathcal{O}_{P,A_1}(\blambda^\bs)$ such that $\bs=(\pi_1(\bs),\bt)\in\mathbb{R}^{|U_1|}\times\mathbb{R}^{|U_2|}$;
\item for any lattice point $\bs\in\mathcal{CO}_{U_1,U_2}(\blambda)$, there exists a unique lattice point $\bt\in\mathcal{C}_{P,A_1}(\blambda^\bs)$ such that $\bs=(\pi_1(\bs),\bt)\in\mathbb{R}^{|U_1|}\times\mathbb{R}^{|U_2|}$.
\end{enumerate}
\end{remark}

We show in the following example that projections of a marked chain polytope do not satisfy the decomposition property in general.

\begin{example}
Consider the linear poset 
$$P=\{b\prec x_3\prec x_2\prec x_1\prec a\}$$
with marked points $A=\{a,b\}$. We fix a marking $\blambda=(\lambda_a,\lambda_b)=(6,0)$. Let $U_1=\{x_1\}$ and $U_2=\{x_2,x_3\}$. 
\par
It is clear that $\bt=(x_1,x_2,x_3)=(0,3,2)$ is in $\mathcal{C}_{P,A}(\blambda)$. If we consider the marked poset $(P,A_2,\bmu)$ with $A_2=A\cup U_2=\{a,x_2,x_3,b\}$ and the marking 
$$\bmu=(\bmu_a,\bmu_{x_2},\bmu_{x_3},\bmu_b)=(6,3,2,0),$$
the point $x_1=3$ is a lattice point in the marked chain polytope associated to $(P,A_2,\bmu)$, but $(3,3,2)\notin\mathcal{C}_{P,A}(\blambda)$.
\end{example}

\section{Normality and Minkowski sum property}\label{sec3}
\subsection{Normality of marked chain-order polytopes}
Notice that by Proposition \ref{Prop:COdialation}, for any $n\in\mathbb{N}$, we have:
$$S(n\,\mathcal{CO}_{U_1,U_2}(\blambda))=S_{U_1,U_2}(n\blambda).$$
We start with the following property: \begin{theorem}\label{Thm:normal}
Let $(P, A, \blambda)$ be a marked poset and $(U_1, U_2)$ be an admissible decomposition of it. Then for any $N\in\mathbb{N}$, $S_{U_1,U_2}(N\blambda)=S_{U_1,U_2}(\blambda)+\ldots+S_{U_1,U_2}(\blambda)$ is the $N$-fold Minkowski sum.
\end{theorem}
Note that Proposition~\ref{Prop:COdialation} deals with all points of the polytope, while here we would like to decompose the lattice points as a sum of lattice points.

\begin{proof}
For $\blambda \in \bz_{\geq 0}^{|A|}$ and $N \geq 2$, it suffices to show that 
\begin{equation}\label{Eq:normal}
S_{U_1,U_2}(N\blambda)=S_{U_1,U_2}(\lceil N/2 \rceil\blambda)+S_{U_1,U_2}(\lfloor N/2 \rfloor\blambda).
\end{equation}
We first prove that (\ref{Eq:normal}) holds when projected onto coordinates in $U_1$, i.e. for $A_1 = U_1 \cup A$:
\begin{equation}\label{Eq:order}
S_{A_1,A}(N\blambda)=S_{A_1,A}(\lceil N/2 \rceil\blambda)+S_{A_1,A}(\lfloor N/2 \rfloor\blambda).
\end{equation}
Let $\bs\in S_{A_1,A}(N\blambda)$, for any $x\in U_1$, we write
$$\bs_x=Nr_x+v_x\ \ \text{for some $r_x\in\mathbb{N}$ and $0\leq v_x<N$.}$$
We define for $x\in U_1$,
$$(\bs^1)_x=\lceil N/2 \rceil r_x + \text{min}\{ v_x, \lceil N/2 \rceil \}\ \text{ and }\ (\bs^2)_x=\lfloor N/2 \rfloor r_x + \text{max}\{ 0, v_x - \lceil N/2 \rceil \}.$$

\noindent
\textbf{Claim.}  $\bs^1\in S_{A_1,A}(\lceil N/2 \rceil\blambda)$ and $\bs^2\in S_{A_1,A}(\lfloor N/2 \rfloor\blambda)$.
\begin{proof}[Proof of the claim]
\begin{enumerate}
\item For any $a\prec x\prec b$ with $a,b\in A$ and $x\in U_1$, we have to show
$$\lceil N/2 \rceil \blambda_a \leq (\bs^1)_x  \leq \lceil N/2 \rceil \blambda_b\ \ \text{and}\ \ \lfloor N/2 \rfloor \blambda_a \leq (\bs^2)_x \leq  \lfloor N/2 \rfloor \blambda_b.$$
Indeed, by definition we have 
\begin{equation}\label{rx}
r_x  \lceil N/2 \rceil \leq (\bs^1)_x \leq  (r_x + 1) \lceil N/2 \rceil  \text{ and } r_x  \lfloor N/2 \rfloor \leq (\bs^2)_x \leq  (r_x + 1) \lfloor N/2 \rfloor.
\end{equation}
By assumption:
$$N \blambda_a \leq \bs_x = N r_x  + v_x \leq N \blambda_b,$$
which implies that $\blambda_a \leq r_x \leq  \blambda_b$. By studying two cases $r_x=\blambda_b$ and $r_x<\blambda_b$ and applying \ref{rx}, we obtain the desired inequalities.
\item For any $x,y\in U_1$ with $x\prec y$ we have to show
$$(\bs^1)_x\leq (\bs^1)_y,\ \ (\bs^2)_x\leq (\bs^2)_y.$$
The assumption $\bs_x\leq \bs_y$ implies $r_x\leq r_y$. By considering two cases $r_x=r_y$ and $r_x<r_y$, it is easy to deduce the inequalities.
\end{enumerate}
\end{proof}

We turn to the proof of (\ref{Eq:normal}). Let $\bs\in S_{U_1,U_2}(N\blambda)$. Then by (\ref{Eq:order}) there exists $\bs^1\in S_{A_1,A}(\lceil N/2 \rceil\blambda)$ and $\bs^2\in S_{A_1,A}(\lfloor N/2 \rfloor\blambda)$ such that $\pi_1(\bs)=\bs^1+\bs^2$. By Lemma \ref{Lem:Decomposition}, there exists a unique $\bt\in \mathcal{C}_{P,A_1}((N\lambda)^{\pi_1(\bs)})$ such that $\bs=(\pi_1(\bs),\bt)$.
\par
By Theorem \ref{Thm:C/O} (2), 
$$\mathcal{C}_{P,A_1}((N\lambda)^{\pi_1(\bs)})=\mathcal{C}_{P,A_1}((\lceil N/2\rceil\lambda)^{\bs^1})+\mathcal{C}_{P,A_1}((\lfloor N/2\rfloor\lambda)^{\bs^2}).$$
We let $\bt=\bt^1+\bt^2$ a corresponding decomposition of $\bt$. 
Again by Lemma \ref{Lem:Decomposition}, $(\bs^1,\bt^1)\in S_{U_1,U_2}(\lceil N/2 \rceil\blambda)$ and $(\bs^2,\bt^2)\in S_{U_1,U_2}(\lfloor N/2 \rfloor\blambda)$, therefore $\bs=(\bs^1,\bt^1)+(\bs^2,\bt^2)$ is the required decomposition. The other inclusion is clear.
\end{proof}

\begin{corollary}
The marked chain-order polytope $\mathcal{CO}_{U_1,U_2}(\lambda)$ is a lattice polytope, so it is a normal polytope.
\end{corollary}

\begin{proof}
The proof of Lemma 11.7 in \cite{FFoL13b} can be applied here.
\par
Pick a point $p\in \mathcal{CO}_{U_1,U_2}(\lambda)$ with rational coordinates. Let $n\in\mathbb{N}$ such that $np$ has integral coordinates. Then $np\in S_{U_1,U_2}(n\blambda)$. By Theorem \ref{Thm:normal}, there exists $p_1,p_2,\ldots,p_n\in S_{U_1,U_2}(\lambda)$ such that 
$$np=p_1+p_2+\ldots+p_n.$$
Hence $p=\frac{1}{n}p_1+\frac{1}{n}p_2+\ldots+\frac{1}{n}p_n$ 
is in the convex hull of $S_{U_1,U_2}(\lambda)$. 
\par
By Remark \ref{Rmk:RationalVertex}, all vertices of $\mathcal{CO}_{U_1,U_2}(\lambda)$ have rational coordinates. This implies that $\mathcal{CO}_{U_1,U_2}(\lambda)$ is contained in the convex hull of $S_{U_1,U_2}(n\blambda)$. Hence $\mathcal{CO}_{U_1,U_2}(\lambda)$ is a lattice polytope and by Theorem \ref{Thm:normal}, it is normal.
\end{proof}

We obtain the normality of the marked order and marked chain polytopes as special cases.  

\begin{corollary}
The marked order polytope $\mathcal{O}_{P,A}(\blambda)$ and the marked chain polytope $\mathcal{C}_{P,A}(\blambda)$ are both normal.
\end{corollary}

\begin{remark}
Up to our knowledge, there is no proof in the literature for the normality of the marked order and marked chain polytopes. For the order polytope, the normality is proved by Stanley \cite{Sta86} by showing the existence of a unimodular triangulation.
\end{remark}

\subsection{Minkowski sum property}

We generalize the Minkowski sum property for marked chain polytopes (Theorem \ref{Thm:C/O} (2)) to marked chain-order polytopes with linearly ordered markings.

\begin{theorem}\label{Thm:Minkowski}
Suppose that $A$ is linearly ordered and $\blambda, \bmu\in\mathbb{Z}_{\geq 0}^{|A|}$ are two markings of $A$. Then the Minkowski sum property holds:
$$
S_{U_1, U_2}(\blambda)  +S_{U_1, U_2}( \mu)  = S_{U_1, U_2}( \blambda+ \bmu).
$$
\end{theorem}
\begin{proof}
Since $A$ is linearly ordered, we may suppose $A = \{ a_1\prec a_2\prec \ldots \prec a_s \}$. For $1\leq i\leq s$, let $\omega^i$ denote the marking of $A$ satisfying:
$$(\omega^i)_{a_j} = \begin{cases} 0, & \text{if }\ 1\leq j < i; \\ 1, & \text{if }\  i \leq j\leq s. \end{cases}$$ 
It suffices to prove the theorem for $\bmu = \omega^i$ and we can moreover assume that for any $1\leq j<i$, $\blambda_{a_j}=0$.
\par
We prove the Minkowski sum property under this assumption. Let $\bs \in S_{U_1, U_2}(\blambda + \omega^i)$. The partial order on $P$ induces a partial (lexicographic) order on  $n S(\mathcal{O}_{A_1,A}(\omega^i))$ and let $\mathbf{m} \in S(\mathcal{O}_{A_1,A}(\omega^i))$ be a maximal element such that for any $x\in U_1$: $\bm_x \leq \bs_x$. \\
\\
\textbf{Claim:} $\pi_1(\bs)-\bm\in\mathcal{O}_{A_1, A}(\blambda)$.
\begin{proof}[Proof of the claim]
We check that $\pi_1(\bs)-\bm$ satisfies all relations in $\mathcal{O}_{A_1, A}(\blambda)$.
\begin{enumerate}
\item For $x\prec y\in U_1$, we need to show that $\bs_x-\bm_x\leq \bs_y-\bm_y$. Since $\bs_x\leq\bs_y$, we claim that it suffices to consider the case $\bm_x<\bm_y$. Indeed, if $\bm_y\leq\bm_x$, then $\bs_x-\bs_y\leq\bm_x-\bm_y$ is always true.
\par
Suppose that $\bm_x<\bm_y$. In this case $\bm_x=0$ and $\bm_y=1$ since $\bm\in\mathcal{O}_{A_1, A}(\omega^i)$; by the maximality of $\bm$, $\bs_x=0$, hence $\bs_x\leq\bs_y$.
\item For $x\prec a_k$ and $a_j\prec y$ where $x,y\in U_1$ and $a_k,a_j\in A$. We need to show that $\bs_x-\bm_x\leq\blambda_{a_k}$ and $\blambda_{a_j}\leq\bs_y-\bm_y$. We prove the first inequality, a similar proof works for the second one.
\par
Since $\bs_x\leq\blambda_{a_k}+(\omega^i)_{a_k}$, we study two cases:
\begin{itemize}
\item the case $k<i$: in this case $\bs_x-\blambda_{a_k}\leq 0$, so $\bs_x-\bm_x\leq\blambda_{a_k}$ is always true;
\item the case $k\geq i$: in this case $\bs_x-\blambda_{a_k}\leq 1$. If $\bs_x=0$, then $\bs_x-\blambda_{a_k}\leq \bm_x$ is always true. If $\bs_x\geq 1$, then $\bm_x=1$ and $\bs_x-\blambda_{a_k}\leq \bm_x$ holds. 
\end{itemize}
\end{enumerate}
\end{proof}

By Lemma \ref{Lem:Decomposition}, it remains to show that 
$$
\mathcal{C}_{P, A_1}((\blambda+\omega^i)^{\pi_1(\bs)})= \mathcal{C}_{P, A_1}(\blambda^{\pi_1(\bs)-\bm})+ \mathcal{C}_{P,A_1}((\omega^i)^{\bm}),
$$
but this follows from Theorem \ref{Thm:C/O} (2).
\end{proof}
 
\begin{corollary}
If $A$ is linearly ordered, then the polyhedral cone $\Cone(U_1,U_2)$ is finitely generated.
\end{corollary}
 
\begin{remark}
The condition \textit{linearly ordered} secures that one can subtract the indecomposable marking $\omega^i$ (\cite{Fou15a}). For non-linearly ordered subset $A$ this is more complicated as the indecomposable markings depend on the decomposition $(U_1, U_2)$.
\end{remark}

\section{Ehrhart equivalence}\label{sec4}
After proving that the marked chain-order polytopes are lattice polytopes, we study the Ehrhart polynomials.

\begin{theorem}\label{Thm:Ehrhart}
Let $(P,A,\blambda)$ be a marked poset. Then all marked chain-order polytopes associated to admissible decompositions are Ehrhart equivalent.
\end{theorem}

That is equivalent to: for any admissible decomposition $(U_1,U_2)$ of the marked poset $(P,A,\blambda)$ and $n \in\mathbb{N}$:
$$
|S_{U_1, U_2} (n\blambda)|  = |S_{P \setminus A, \emptyset}( n\blambda)| .
$$

\begin{remark}\label{rem-abs}
The theorem holds for the special case of marked order and marked chain polytopes (see Theorem \ref{Thm:C/O} (1)), i.e., for the datum $(U_1,U_2)=(\emptyset, P \setminus A)$. This has been shown in \cite{ABS11} by constructing an explicit piecewise linear, affine bijection that respects integral points. This map $\varphi: \mathbb{R}^{|P \setminus A|} \longrightarrow \mathbb{R}^{|P\setminus A|}$ defined for any point $(x_p) \in \mathcal{O}_{P,A}(\blambda)$ by
$$
\varphi(x)_p = \operatorname{min }  \{x_p - x_q \mid p > q, q \notin A\} \cup \{x_p - \lambda_q \mid p > q, q \in A \}.
$$
Then $\varphi(x)_p \in  \mathcal{C}_{P, A}(\blambda)$ and in fact this is a bijection.
\end{remark}
\begin{proof}[Proof of the theorem]
It suffices to show that for any $n\in\mathbb{N}$, 
$$|S_{U_1, U_2} (n\blambda)|  = |S_{P \setminus A, \emptyset}( n\blambda)|.$$ 
Denote as before $A_1 = A \cup U_1$.
By Lemma \ref{Lem:pr1} and \ref{Lem:Decomposition}, for any marking $\mu\in\mathbb{Z}_{\geq 0}^{|A|}$ of $A$,
$$|S_{U_1,U_2}(\bmu)|=\sum_{\bs\in\mathcal{O}_{A_1,A}(\bmu)}|S(\mathcal{C}_{P,A_1}(\bmu^{\bs}))|,$$
and
$$|S_{P\setminus A,\emptyset}(\bmu)|=\sum_{\bs\in\mathcal{O}_{A_1,A}(\bmu)} |S(\mathcal{O}_{P,A_1}(\bmu^{\bs}))|.$$
Then the theorem holds by Theorem \ref{Thm:C/O} (1), since 
$$ |S(\mathcal{C}_{P,A_1}(\bmu^{\bs}))| =  |S(\mathcal{O}_{P,A_1}(\bmu^{\bs}))|.$$
\end{proof}

\begin{remark}
In fact this theorem is true for layered marked poset polytopes as will be shown in \cite{FF15}.
\end{remark}

\section{Isomorphic polytopes}\label{sec5}
Let $(P,A,\blambda)$ be a marked poset and $(U_1,U_2)$, $(V_1,V_2)$ be two different admissible decompositions. It is natural to ask whether $\mathcal{CO}_{U_1,U_2}(\blambda)$ and $\mathcal{CO}_{V_1,V_2}(\blambda)$ are unimodular equivalent. We will answer this question in this section.

\subsection{Regular marked poset}

We recall the concept of regular marked posets in \cite{Fou15a}.

\begin{definition}\label{def:regular}
A marked poset $(P,A,\blambda)$ is called \emph{regular} if
\begin{enumerate}
\item there does not exist $a,b\in A$ such that $a\ra b$;
\item for any $a\neq b\in A$, $\blambda_a\neq\blambda_b$;
\item if for $a\in A$ and $x\in P\setminus A$, $a\ra x$, then there does not exist $b\in A$ such that $b\prec x$ and $\blambda_a<\blambda_b$;
\item if for $a\in A$ and $x\in P\setminus A$, $x\ra a$, then there does not exist $b\in A$ such that $x\prec b$ and $\blambda_b<\blambda_a$.
\end{enumerate}
\end{definition}

The following proposition is proved in Section 3 of \cite{Fou15a}.

\begin{proposition}[\cite{Fou15a}]\label{Prop:regular}
For any marked poset $(P,A,\blambda)$ there exists a regular marked poset $(P^r,A^r,\blambda^r)$ such that:
\begin{enumerate}
\item $\mathcal{O}_{P,A}(\blambda)$ is unimodular equivalent to $\mathcal{O}_{P^r,A^r}(\blambda^r)$.
\item $\mathcal{C}_{P,A}(\blambda)$ is unimodular equivalent to $\mathcal{C}_{P^r,A^r}(\blambda^r)$.
\end{enumerate}
In particular, the number of facets of these polytopes coincides.
\end{proposition}

\begin{remark}\label{Rmk:regular}
In the proof of the proposition, to obtain the regular marked poset, we either retract a chain between two marked vertices having the same marking or retract some marked vertices.
\end{remark}

Let $(U_1,U_2)$ be an admissible decomposition of the marked poset $(P,A,\blambda)$. The following proposition is clear by the remark above.

\begin{proposition}\label{Prop:regulardec}
There exists a regular marked poset $(P^r,A^r,\blambda^r)$ with an admissible decomposition $(U_1^r,U_2^r)$ such that $\mathcal{CO}_{U_1,U_2}(\blambda)$ is unimodular equivalent to $\mathcal{CO}_{U_1^r,U_2^r}(\blambda^r)$.
\end{proposition}

\subsection{Facets of marked chain-order polytopes}

By Proposition \ref{Prop:regulardec}, in aim of counting facets, we may suppose that the marked poset $(P,A,\blambda)$ is regular. We keep this hypothesis in this paragraph.
\begin{proposition}\label{Prop:facetcount}
The number of facets in $ \mathcal{CO}_{U_1, U_2}( \blambda)$ is 
\begin{small}
\begin{eqnarray*}
& & \#\{p\ra q\ |\ p,q\in A_1\}\\
&+&\# U_2\\
&+& \#\{\text{maximal chains $q\rightarrow p_1 \rightarrow \ldots \rightarrow p_s \rightarrow r$ where $s\geq 1$, $q,r\in A_1$ and $p_i\in U_2$}\}.
\end{eqnarray*}
\end{small}
\end{proposition}

Suppose that $(V_1,V_2)$ is another admissible decomposition of $(P,A,\blambda)$ such that 
$$V_2 \cap \St(P) = (U_2 \cap \St(P) )\cup \{q\},$$ 

\begin{corollary}\label{Cor:SHLfacet}
The difference of the number of the facets in $\mathcal{CO}_{V_1, V_2}(\blambda)$ and the number of facets in $\mathcal{CO}_{U_1, U_2}( \blambda)$ is $(|q \rightarrow | - 1)(| \rightsquigarrow q| -1)$ .
\end{corollary}

\subsection{Unimodular equivalence}

The following theorem answers the question in the beginning of this section.

\begin{theorem}\label{Thm:unimodular}
Let $(U_1,U_2)$ and $(V_1,V_2)$ be two different admissible decompositions of the marked poset $(P,A,\blambda)$. If $U_1 \cap \St(P) = V_1 \cap \St(P)$, then the chain-order polytopes $\mathcal{CO}_{U_1, U_2}( \blambda)$ and $\mathcal{CO}_{V_1, V_2}(\blambda)$ are unimodular equivalent.\end{theorem}

\begin{proof}
Notice that $U_1 \cap \St(P) = V_1 \cap \St(P)$ is equivalent to $U_2 \cap \St(P) = V_2 \cap \St(P)$.
\par
By the hypothesis $U_1 \cap \St(P) = V_1 \cap \St(P)$, we may assume that $U_1 = V_1\cup \{p\}$, where $p \notin \St(P)$. Since $p \notin \St(P)$, there are exactly two cases to be considered:
\begin{enumerate}
\item There is exactly one maximal chain ending in $p$, say $a\prec q_t\prec\ldots\prec q_1\prec p$, with $q_i \in V_1$ and $a\in A$. Then we define a map 
$$
\Psi:\mathcal{CO}_{V_1, U_2 \cup \{p\}}( \blambda) \longrightarrow \mathcal{CO}_{U_1, U_2}( \blambda)
$$
by 
$$
\Psi(x)_q := \begin{cases} x_q, & \text{ if } q \neq p; \\ x_p -  x_{q_1} - \ldots - x_{q_t} - \lambda_a, & \text{ else. } \end{cases}
$$
This is certainly a unimodular equivalence once we have checked its bijectivity. For this we will see that facets are mapped to facets. For this let $p \prec q$, then the facet $x_p\leq x_q$ is mapped to the facet $x_{q_1} + \ldots  + x_{q_t} + x_p \leq x_q - \lambda_a$ (since $\Psi(x)_{q_1} + \ldots + \Psi(x)_{q_t} + \Psi(x)_p = x_p - \lambda_a \leq x_q - \lambda_a$). The facet $x_{q_1} + \ldots + x_{q_t} \leq x_q - \lambda_a$ is mapped to the facet $x_q \geq 0$.
\item There is exactly one element $r \in U_1$ covering $p$. Then we define a map 
$$
\Psi:\mathcal{CO}_{V_1, U_2 \cup \{p\}}(\blambda) \longrightarrow \mathcal{CO}_{U_1, U_2}( \blambda)
$$
by 
$$
\Psi(x)_q := \begin{cases} x_q, & \text{ if } q \neq p; \\  x_r  - x_p, & \text{ else. } \end{cases}
$$
Again, as before, it remains to check that we have a bijection on facets. Let $a\prec q_s\prec\ldots\prec q_1\prec p$ be a maximal chain ending in $p$, then the facet $x_{q_1} + \ldots +x_{q_s} \leq x_p - \lambda_a$ is mapped to $x_{q_1} + \ldots + x_{q_s} + x_p \leq x_r - \lambda_a$, while the facet $x_p \leq x_r$ is mapped to the facet $x_p \geq 0$.
\end{enumerate}
Thus in both cases we have defined a unimodular equivalence.
\end{proof}

\begin{corollary}\label{Cor:unimodular}
Let $(U_1,U_2)$ and $(V_1,V_2)$ be two different admissible decompositions of the marked poset $(P,A,\blambda)$ such that $U_1\subset V_1$. The chain-order polytopes $\mathcal{CO}_{U_1, U_2}( \blambda)$ and $\mathcal{CO}_{V_1, V_2}(\blambda)$ are unimodular equivalent if and only if
$$U_1 \cap \St(P) = V_1 \cap \St(P).$$
\end{corollary}

\begin{proof}
The if-part is proved in Theorem \ref{Thm:unimodular}; under the assumption $U_1\subset V_1$, the only if-part holds by Proposition \ref{Prop:facetcount}.
\end{proof}

\section{The Stanley-Hibi-Li conjecture}\label{sec6}
We recall here a conjecture by Stanley, Hibi and Li for chain and order polytopes, \cite{Sta86, HL12a} (see the generalization for marked chain and marked order polytopes in \cite{Fou15a}). For a given $N$-dimensional polytope $P$, we denote
$$
f_i(P) = \# \{ i\,\text{--\,dimensional faces of } P \}
$$
and define the $f$-vector of $P$ to be $f(P):= (f_0(P), \ldots, f_{N}(P))$.
\begin{conjecture}[Stanley, Hibi-Li]\label{hl-conjecture}
Let $(P, A, \blambda)$ be a marked poset. Then for all $i = 0 ,1, \ldots, |P\setminus A|$:
$$f_i(\mathcal{C}_{P, A}(\blambda)) \geq f_i(\mathcal{O}_{P, A}(\blambda)).$$
\end{conjecture}
\begin{remark} 
Here are some remarks around the history and the known cases of the conjecture.
\begin{enumerate}
\item The original conjecture was stated for chain and order polytopes only (\cite{Sta86, HL12a}). The conjecture has been stated for marked chain and marked order polytopes in \cite{Fou15a}.
\item The ($|P\setminus A|-1$)-dimensional case (the number of facets) has been shown for chain and order polytopes in \cite{HL12a} and for marked chain and marked order polytopes in \cite{Fou15a}.
\item The $0$-dimensional case for chain and order polytopes has been shown in \cite{Sta86}, where in fact is was shown that the number of vertices for both polytopes is the same. The $0$-dimensional case for marked chain and order polytopes is still open.
\end{enumerate}
\end{remark}

We would like to state a refinement of the conjecture above.

\begin{conjecture}
Let $(P,A,\blambda)$ be a marked poset, $(U_1, U_2)$ and $(V_1, V_2)$ be two different admissible decompositions with $U_1 \subset V_1$. Then for any $i = 0, 1,\ldots, |P \setminus A|$:
$$
f_i(\mathcal{CO}_{V_1, V_2}(\blambda)) \leq f_i(\mathcal{CO}_{U_1, U_2}(\blambda)).
$$
\end{conjecture}
\begin{remark}
\begin{enumerate}
\item The case $(U_1,U_2) = (\emptyset,P\setminus A)$ and $(V_1,V_2) = (P\setminus A,\emptyset)$ is precisely Conjecture~\ref{hl-conjecture}.
\item The conjecture in the case $i=|P \setminus A|-1$, i.e., the number of facets, holds by Corollary~\ref{Cor:SHLfacet}.
\end{enumerate}
\end{remark}

\section{Example}\label{sec7}

In this section we apply the construction of marked chain-order polytopes to the following poset $P_n$, called the Gelfand-Tsetlin poset: 
\begin{itemize}
\item the set of vertices of $P_n$ is $\{p_{i,j}\mid 0\leq i\leq j\leq n\}$;
\item the set of cover relations in $P_n$ (\emph{i.e.}, the edges in the corresponding Hasse diagram) is: for any $1\leq i\leq j\leq n$, 
$$p_{i-1,j}\ra p_{i,j}\ra p_{i-1,j-1}.$$
\end{itemize}

The marking subset $A_n$ is the linearly ordered set 
$$
A_n = \{ p_{0,0} \prec  p_{0,1} \prec \ldots \prec  p_{0,n} \}.
$$
We fix a marking $\blambda=(\blambda_0,\blambda_1,\ldots,\blambda_n)$ where $\blambda_k:=\blambda(p_{0,k})$, satisfying 
\begin{equation}\label{Eq:lambda}
\blambda_0\geq\blambda_1\geq\ldots\geq\blambda_n.
\end{equation}

\begin{example}\label{gt-pattern}
We consider the case $n=4$: the poset $P_4$ is the following:
\\
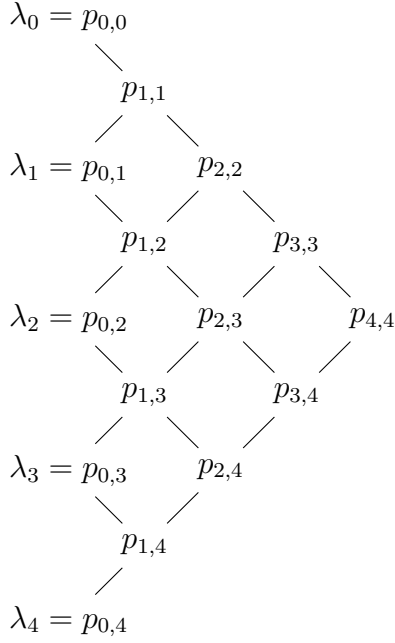
\begin{center}
\begin{tikzpicture}
  \node (a) at (0,0) {$\blambda_0=p_{0,0}$};
  \node (b) at (0,-2) {$\blambda_1=p_{0,1}$};
  \node (c) at (0,-4) {$\blambda_2=p_{0,2}$};
  \node (d) at (0,-6) {$\blambda_3=p_{0,3}$};
  \node (e) at (0,-8) {$\blambda_4=p_{0,4}$};
  \node (f) at (1,-1) {$p_{1,1}$};
  \node (g) at (1,-3) {$p_{1,2}$};
  \node (h) at (1,-5) {$p_{1,3}$};
  \node (i) at (1,-7) {$p_{1,4}$};
  \node (j) at (2,-2) {$p_{2,2}$};
  \node (k) at (2,-4) {$p_{2,3}$};
  \node (l) at (2,-6) {$p_{2,4}$};
  \node (m) at (3,-3) {$p_{3,3}$};
  \node (n) at (3,-5) {$p_{3,4}$};
  \node (o) at (4,-4) {$p_{4,4}$};
  \draw (a) -- (f) -- (j) -- (m) -- (o) 
  (b) -- (g) -- (k) -- (n)
  (c) -- (h) -- (l) 
  (d) -- (i)
  (b) -- (f)
  (c) -- (g) -- (j)
  (d) -- (h) -- (k) -- (m)
  (e) -- (i) -- (l) -- (n) -- (o);
\end{tikzpicture}
\end{center}
\end{example}

In the following lemma we list some basic properties of this marked poset $(P_n,A_n,\blambda)$.

\begin{lemma}
\begin{enumerate}
\item If $\blambda_0>\blambda_1>\ldots>\blambda_n$, then $(P_n,A_n,\blambda)$ is a regular marked poset.
\item The star elements in $(P_n,A_n,\blambda)$ are $\St(P) = \{ p_{i,j} \mid 1 \leq i < j < n \}$.
\end{enumerate}
\end{lemma}

By this lemma, results in the previous sections can be applied to the polytopes associated to this marked poset. Let $(U_1,U_2)$ be an admissible decomposition of $(P_n,A_n,\blambda)$.

\begin{corollary}
\begin{enumerate}
\item The chain-order polytope $\mathcal{CO}_{U_1,U_2}(\blambda)$ is a normal polytope.
\item For any two markings $\blambda,\bmu$ of $A_n$ satisfying (\ref{Eq:lambda}), the lattice points in the chain-order polytope satisfy the Minkowski property:
$$S_{U_1,U_2}(\blambda)+S_{U_1,U_2}(\bmu)=S_{U_1,U_2}(\blambda+\bmu).$$
\item The polyhedral cone $\Cone(U_1,U_2)$ is finitely generated.
\end{enumerate}
\end{corollary}

We turn to count the number of lattice points.

\begin{corollary}
Let $\lambda = (\lambda_0, \ldots, \lambda_n)$ be a marking satisfying (\ref{Eq:lambda}). Then for any admissible decomposition $(U_1,U_2)$, the number of lattice points in the marked chain-order polytope $\mathcal{CO}_{U_1,U_2}(\blambda)$ is given by 
$$
\prod_{0 \leq i \leq j \leq n}\frac{ (\blambda_i-\blambda_j+j-i+1)}{j-i+1}.
$$
\end{corollary}

\begin{proof}
By Theorem \ref{Thm:Ehrhart}, it suffices to count the lattice points in the marked order polytope $\mathcal{O}_{P_n,A_n}(\blambda)$, which is shown in \cite{ABS11} to be the dimension of the irreducible representation $V(\blambda)$ of $\lie {sl}_{n+1}$ associated to $\blambda$. Then Weyl's dimension formula applies.
\end{proof}

We estimate the number of non-isomorphic chain-order polytopes. We define
$$\mathcal{M}_{P_n,A_n}(\blambda)=\{\mathcal{CO}_{U_1,U_2}(\blambda)\mid (U_1,U_2)\text{ is an admissible decomposition }\}/\sim,$$
where two polytopes $Q_1\sim Q_2$ if and only if they are unimodular equivalent.

\begin{proposition}\label{Prop:count}
Suppose $\blambda$ is regular, then we have $\#\mathcal{M}_{P_n,A_n}(\blambda)\leq 2^{n-2}$.
\end{proposition}

\begin{proof}
By Theorem \ref{Thm:unimodular}, the cardinality of $\mathcal{M}_{P_n,A_n}(\blambda)$ is no more than the number of partitions of $\St(P_n)$ into two subsets $S_1$ and $S_2$ such that there does not exist $s_1\in S_1$ and $s_2\in S_2$ satisfying $s_1\prec s_2$, which will be proved to be $2^{n-2}$. We call such a partition admissible.
\par
We argue by induction on $n$. The difference between the star elements $\St(P_n)$ and $\St(P_{n-1})$ is
$$R=\{p_{1,n-1},\, p_{1,n-2},\,\ldots,\, p_{n-2,n-1}\}.$$
Let $S_1$, $S_2$ be an admissible partition of $\St(P_n)$. Since the set $R$ is linearly ordered, there exists $0\leq k\leq n-2$ such that 
$$S_1\cap R=\{p_{k+1,n-1},\,\ldots,\,p_{n-2,n-1}\},\ \ S_2\cap R=\{p_{1,n-1},\,\ldots,\,p_{k,n-1}\}.$$
We count the numbers of admissible partitions $S_1$, $S_2$ of $\St(P_n)$ satisfying this property. Suppose that $k\geq 2$.
\par
The hypothesis on $S_1$ and $S_2$ ensures that
$$\{p_{k+1,n-1},p_{k,n-2},\ldots,p_{1,n-k-1},p_{k+2,n-1},\ldots,p_{1,n-k-2},\ldots,p_{n-2,n-1},\ldots,p_{2,1}\}\subset S_1.$$
It suffices to consider the admissible partitions of 
$$T=\{p_{1,n-2},p_{2,n-2},p_{1,n-3},\ldots,p_{k-1,n-2},\ldots,p_{1,n-k}\},$$
which can be looked as the star elements in $P_{k+1}$. By induction hypothesis, there are $2^{k-1}$ such admissible partitions.
\par
When $k=0$ or $k=1$, there exists a unique partition of $T$. Therefore the number of different admissible partitions of $\St(P_n)$ is
$$1+\sum_{k=0}^{n-3}2^k=2^{n-2}.$$
\end{proof}

\end{document}